\documentclass[10pt]{amsart}
\usepackage{graphicx}
\usepackage{amssymb}
\usepackage{amsmath,amsthm}
\usepackage{hyperref}
\usepackage{amsthm}
\usepackage[english]{babel}
\usepackage{mathrsfs}

\newtheorem{thm}{Theorem}[section]
\newtheorem*{thm*}{Theorem}
\newtheorem{cor}[thm]{Corollary}
\newtheorem*{cor*}{Corollary}

\newtheorem{prop}[thm]{Proposition}
\theoremstyle{definition}
\newtheorem{defn}[thm]{Definition}

\title[A Unique Prime Decomposition Result for Wreath Product Factors]{A Unique Prime Decomposition Result for Wreath Product Factors}

\author{J. Owen Sizemore}
\address{J. Owen Sizemore, UCLA, Math Sciences Building, Los Angeles, CA 90095-1555}
\email{sizemore@math.ucla.edu} \subjclass{} \subjclass{}
\author{Adam Winchester}
\address{Adam Winchester, UCLA, Math Sciences Building, Los Angeles, CA 90095-1555}
\email{lagwadam@math.ucla.edu} \subjclass{} \subjclass{}

 \dedicatory{}
\date{\today}
\begin{document}

\begin{abstract}We use malleable deformations combined with spectral gap rigidity theory, in the
framework of Popa's deformation/rigidity theory to prove unique
tensor product decomposition results for II$_1$ factors arising as
tensor product of wreath product factors. We also obtain a similar
result regarding measure equivalence decomposition of direct
products of such groups.
\end{abstract}

\maketitle \tableofcontents

%%%%%%%%%%%%%%%%%%%%%%%%%%%%%%%%%%%%%%%%%%%%%%%%%%%%

\section*{Introduction}

A major goal of the study of II$_1$ factors is the classification of
these algebras based on the ``input data'' that goes into their
construction. A significant landmark was the result, due to Connes
\cite{C76}, that all amenable II$_1$ factors are isomorphic.
However, in the non-amenable realm there is a much greater variety,
and a striking classification theory has developed.

One thrust of this research is to determine if some algebra which,
\textit{a priori}, is constructed in one manner, can be obtained in
some other manner. For example, if we have a II$_1$ factor that we
know to be a free product of two II$_1$ factors, is it also possible
to be the tensor product of two (possibly different) II$_1$ factors?

In this vein we study whether certain factors can be written as a
tensor product in two distinct ways. Such results go back to the
study of prime factors, (ie. a factor which cannot be written as the
tensor product of two other II$_1$ factors.) The first result was
obtained by Popa in, \cite{Po83}, where he showed that the
group von Neumann algebra of an uncountable free group is prime.

Later, in \cite{Ge98}, Ge proves that all group factors coming from
finitely generated free groups are prime. Using $C^*$ techniques
this was greatly generalized by Ozawa, \cite{Oz}, to show that all
i.c.c. Gromov hyperbolic groups give rise to prime factors. Also,
using his deformation/rigidity theory, Popa showed in
\cite{P-gap} that all II$_1$ factors arising from the Bernoulli
actions of  nonamenable groups are prime. Further, Peterson used his
derivation approach to deformation/rigidity (\cite{Pe06}) to prove that any II$_1$ factor coming from a
countable group with positive first $l^2$-betti number is also
prime. Finally we should also note that using Popa's
deformation/rigidity theory, Chifan and Houdayer,
\cite{C-H}, gave many more examples of prime II$_1$-factors coming
from amalgamated free products.

A natural question about prime factors is whether
a tensor product of a finite number of such factors $P_1, P_2, ..., P_n$,
has a ``unique prime factor decomposition'', i.e., if $P_1 \overline{\otimes}  ... \overline{\otimes} P_n
= Q_1 \overline{\otimes} .... \overline{\otimes} Q_m$, for some other prime factors $Q_j$,
forces $n=m$ and $P_i$ unitary conjugate to $Q_i$, modulo some permutation of indices
and modulo some ``rescaling'' by appropriate amplifications of the prime factors involved.
A first such result was obtained
by Ozawa and Popa in \cite{OP03}, where a combination of $C^*$ techniques from \cite{Oz}
and intertwining techniques from \cite{P1} is used to show that any
II$_1$ factor arising from a tensor product of hyperbolic group  factors
has  such a unique tensor product decomposition.

In this paper we prove an analogous unique prime factor decomposition result for tensor products of
wreath product II$_1$ factors. More precisely,  we prove the following result:

\begin{thm}\label{maintheorem}
Let $A_1, \dots, A_n$ be non-trivial amenable groups; $H_1, \dots, H_n$ be
non-amenable groups; and $Q_1, \dots, Q_k$ be diffuse von Neumann algebras
such that
\begin{equation*}
M = L(A_1 \wr H_1) \overline{\otimes} \dots \overline{\otimes} L(A_n \wr H_n) = Q_1 \overline{\otimes} \dots \overline{\otimes} Q_k
\end{equation*}
If $k \geq n$, then $n=k$, and after permutation of indices we have that $L(A_i \wr
H_i) \simeq Q_i^{t_i}$ for some positive numbers $t_1, t_2, \dots
t_n$ whose product is $1$.
\end{thm}

Also we have a natural generalization of this theorem  to unique
measure-equivalence decomposition results of finite products of
wreath product groups. Such results were achieved for products of
groups of the class $\mathcal{C}_{reg}$ by Monod and Shalom (Theorem
1.16 in \cite{MS06}), for products of bi-exact groups by Sako
(Theorem 4 in \cite{Sa09}), and for products of groups in
$\mathcal{QH}_{reg}$ by Chifan and Sinclair (Corollary C in
\cite{CS10}.)

\begin{cor}\label{maintheorem2}
Let $A_1, \dots, A_n$ be non-trivial amenable groups; $H_1, \dots,
H_n$ be non-amenable groups; and $K_1, \dots, K_m$ be groups such
that
\begin{equation*}
A_1 \wr H_1 \times \dots \times A_n \wr H_n\simeq_{ME}K1 \times \dots \times K_m
\end{equation*}
If $m \geq n$, then $n=m$, and after permutation of indices we have
that $A_i \wr H_i \simeq_{ME} K_i$.
\end{cor}

We prove these results by using deformation/rigidity theory. More
precisely, we use the malleable deformation for wreath product group
factors in \cite{cps}, combined with Popa's spectral gap rigidity
and intertwining by bimodules techniques.

\vskip .5in

 \noindent\textbf{Acknowledgements:} We would like to
thank Sorin Popa for suggesting this problem and his encouragement
throughout.

\section{Preliminaries}

\noindent\textbf{Intertwining by Bimodules:} Let us recall Popa's
intertwining by bimodules technique. This is a crucial tool for
locating subalgebras of II$_1$-factors, and is summed up in the
following theorem:

\begin{thm}[Popa, \cite{P1}]\label{intertwinedef}Let $P,Q\subset M$ be
 finite von Neumann algebras. Then the following are
equivalent:
\begin{enumerate}
  \item There exists nonzero projections $p\in P, q\in Q,$ a nonzero
  partial isometry $v\in M,$ and a *-homomorphism
  $\varphi:pPp\rightarrow qQq$ such that $vx=\varphi(x)v,
  \forall x\in pPp$
  \item There is a sub-$P-Q$-bimodule $\mathcal{H}\subset L^2(M)$ that
  is finitely generated as a right $Q$-module.
  \item There is no sequence $u_n\in\mathcal{U}(P)$ such that
  \begin{equation*}
      lim_{n\rightarrow\infty}\|E_Q(xu_ny)\|_2\rightarrow0, \forall x,y\in M
  \end{equation*}

\end{enumerate}

\end{thm}

If any of the above conditions hold we say that \textit{a corner of
P embedds in Q inside M}, denoted $P\prec_MQ$.

 Following \cite{OP1} we have the following definition:
\begin{defn}
Let $P, Q\subset M$ be finite von Neuman algebras.
We say that \textit{$P$ is amenable over $Q$ inside M}, which we denote
$P\lessdot_MQ$, if there is a $P$-central state, $\varphi$, on
$\langle M, e_Q\rangle$ such that $\varphi|_M=\tau$, where $\tau$ is
the trace on $M$.
\end{defn}

Let us note that by Theorem 2.1 in \cite{OP1} $P\lessdot_MQ$ is
equivalent to $L^2(P)\prec\bigoplus L^2(\langle M, e_Q\rangle)$ as
$P$-bimodules. Further, if $P\prec_M Q$ then $L^2(M)$ contains a sub
$P$-$Q$-module, $\mathcal{H}$, that is finitely generated as a right
$Q$ module. Therefore, the projection onto this module will commute
with the right action of $Q$ and will have finite trace. Therefore,
it will be a vector in $L^2(\langle M, e_N\rangle).$ Further, it
will also commute with $P$, so if we look at $L^2(\langle M,
e_N\rangle)$ as a $P$-bimodule, it will contain a central vector.
Since strong containment implies weak containment we get the
following observation.

\begin{prop}\label{intertwineimpliesamenableover}
Let $P,Q\subset M$ be von Neumann algebras. If $P\prec_M Q$ then
$P\lessdot_M Q$.
\end{prop}

\noindent\textbf{Deformation of Wreath Products:} Let $A$ and $H$ be
discrete groups. Then following standard notation we let $A\wr
H=A^H\rtimes H$ denote the standard wreath product. Throughout this
paper we will consider trace preserving actions of $A\wr H$ on a
finite von Neuman algebra $N$, and we consider the resulting crossed
product algebra $M=N\rtimes A\wr H$.

Let $\tilde{A}= A\ast\mathbb{Z}$. If we let $u\in L(\tilde{A})$
denote the Haar unitary that generates $L(\mathbb{Z})$ then for
every $t\in \mathbb {R}$, we define $u^t\doteq\exp(ith)\in L\mathbb
{Z}$. This allows us to define $\theta_t\in
\text{Aut}(L(\tilde{A}))$ by $\theta_t(x)=u^tx(u^*)^t$. By applying
this automorphism in each coordinate we can get an automorphism of
$L(\tilde{A}^H)$. Since the action of $H$ is by permuting the
coordinates, it commutes with $\theta_t$ and so we can extend it to
$L(\tilde{A}\wr H)$. If we now declare that the Haar unitaries in
each coordinate do not act on the algebra $N$, then we can extend to
an automorphism, which we still denote by $\theta_t$ of
$\tilde{M}=N\rtimes\tilde{A}\wr H$.

I is easy to see that $\lim_{t\rightarrow 0}\|u^t-1\|_2=0$ and hence
we have $\lim_{t\rightarrow0}\|\theta_t(x)-x\|_2=0$ for all $x\in
\tilde{M}$. Therefore, the path $(\theta_t)_{t\in\mathbb R}$ is a
deformation by automorphisms of $\tilde M$.

Next we show that $\theta_t$ admits a ``symmetry'', i.e. there
exists an automorphism $\beta$ of $\tilde{M}$ satisfying the
following relations:

$$\beta^2=1,\text{ } \beta_{|_M}=id_{|_M},\text{ }
\beta\theta_t\beta= \theta_{-t}, \text{ for all } t\in \mathbb {R}.
$$

\noindent To see this, first define $\beta_{|_{LA^I}}=id_{|_{LA^I}}$
and then for every $h\in H$ we let $(u)_h$ be the element in
$L\tilde{A}^H$ whose $h^{th}$-entry is $u$ and $1$ otherwise. On
elements of this form we define $\beta((u)_h)=(u^*)_h$, and since
$\beta$ commutes with the actions of $H$ on $A^H$, it extends to an
automorphism of $L(\tilde{A}\wr H)$ by acting identically on $L(H)$.
Finally, the automorphism $\beta$ extends to an automorphism of
$\tilde{M}$, still denoted by $\beta$, which acts trivially on $A$.

Let us note that, with this choice of $\beta$, $\theta_t$ is an
\textit{s-malleable deformation} of $\tilde{M}$ in the sense of
\cite{P1}. In fact, this is the same deformation that the first
author used in \cite{cps}, and is inspired by similar free malleable
deformations in \cite{Po01a, IPP, Io06}, so we refer to this
previous work for additional discussion.

%%%%%%%%%%%%%%%%%%%%%%%%%%%%%%%%%%%%%%%%%%%%%%%%%%%%

\section{Intertwining Techniques for Wreath Products}

In this section we prove the necessary intertwining results for
II$_1$ factors arising from wreath product groups that we will need
in order to prove our desired uniqueness of of tensor product
decomposition.

The following proposition is a relative version of Lemma 4.2 in
\cite{cps}, and will follow a similar proof.

\begin{prop}\label{ultraproduct} Let $N$ be a finite von Neumann algebra. Let $A, H$ be
groups with $A$ non-trivial amenable and $H$ non-amenable. Let
$Q\subset N\rtimes A\wr H = M$ be an inclusion of von Neumann
algebras. Assume $Q$ is not amenable over $N$ inside $M$ then
$Q'\cap \tilde{M}^\omega\subseteq M^\omega$.

\end{prop}

\begin{proof}
As mentioned above this proof follows closely the proof of Lemma 4.2
in \cite{cps} as well as Lemma 5.1 in \cite{P-gap} and other similar
results in the literature.

We will prove the contrapositive so let us assume that $Q'\cap
\tilde{M}^\omega\nsubseteq M^\omega$ Then proceeding as in Lemma 5.1
in \cite{P-gap} We see that

$$L^2(Q)\prec L^2(\tilde{M})\ominus L^2(M)$$

as $Q$-bimodules. Now we decompose $L^2(\tilde{M})\ominus L^2(M)$ as
an $M$-bimodule.

One can see that the above $M$-bimodule can be written as a direct
sum of $M$-bimodules $\overline{M\tilde{\eta}_sM}^{\|\cdot\|_2}$ ,
where the cyclic vectors $\tilde{\eta}_s$ correspond to an
enumeration of all elements of $\tilde{A}^H$ whose non-trivial
coordinates start and end with non-zero powers of $u$.

Next, for every $s$, we denote by $\eta_s$ the element of $A^H$ that
remains from $\tilde{\eta}_s$ after deleting all nontrivial powers
of $u$. Also for every $s$ let $\Delta_s\subset H$ be the support of
$\tilde{\eta}_s$ and observe that if
$\textit{Stab}_H(\tilde{\eta}_s)$ denotes the stabilizing group of
$\tilde{\eta}_s$ inside $H$ then we have $\textit{Stab}_H$
$(\tilde{\eta}_s)(H\setminus\Delta_s)\subset H\setminus\Delta_s$.

 Hence we can consider the
von Neumann algebra $K_s =
N\rtimes(A\wr_{H\setminus\Delta_s}\textit{Stab}_H$
$(\tilde{\eta}_s))$ and using similar computations as in Lemma 5.1
of \cite{P-gap}, one can easily check that the map
$x\tilde{\eta}_sy\rightarrow x\eta_se_{K_s}y$ implements an
$M$-bimodule isomorphism between
$\overline{M\tilde{\eta}_sM}^{\|\cdot\|_2}$ and $L^2(\langle M,
e_{K_s}\rangle).$

Therefore, as $M$-bimodules, we have the following isomorphism

$$L^2(\tilde{M})\ominus L^2(M) = \bigoplus L^2(\langle M, e_{K_s}\rangle).$$

Thus we can get the following weak containment of $Q$-bimodules

$$L^2(Q)\prec\bigoplus L^2(\langle M, e_{K_s}\rangle).$$

 Notice that, since $\Delta_s$ is finite, and the action of $H$
on itself is free, then $\textit{Stab}_H(\tilde{\eta}_s)$ is finite
for all $s$. Also, since $A$ is an amenable group we have that
$K_s\lessdot_NN$ for all $s$. Thus for all $s$ we have the following
weak containment of $K_s$-bimodules

$$L^2(K_s)\prec \bigoplus L^2(\langle K_s, e_N\rangle)\simeq\bigoplus L^2(K_s)\otimes_NL^2(K_s)$$

Now if we induce to $M$-bimodules and restrict to $Q$-bimodules and
use continuity of weak containment under induction and restriction
we get the following inclusions of $Q$-bimodules:

\begin{eqnarray*}
L^2(Q)
&\prec&\bigoplus L^2(\langle M, e_{K_s}\rangle)\\
&\simeq&\bigoplus L^2(M)\otimes_{K_s}L^2(K_s)\otimes_{K_s}L^2(M)\\
&\prec&\bigoplus L^2(M)\otimes_{K_s}L^2(K_s)\otimes_NL^2(K_s)\otimes_{K_s}L^2(M)\\
&\simeq&\bigoplus L^2(M)\otimes_NL^2(M)\\
&\simeq&\bigoplus L^2(\langle M, e_N\rangle)
\end{eqnarray*}

Thus $Q\lessdot_MN$

\end{proof}

We finish this section with a final theorem which allows us to
locate regular subfactors with large commutant.

\begin{thm}\label{nAmenOverEmbedd}
Let $N$ be a finite von Neumann algebra. Let $A$ and $H$ be groups
with $A$ non-trivial amenable and $H$ non-amenable. Let $Q\subset
N\rtimes A\wr H = M$ be a subalgebra that is not amenable over $N$.
Let $P=Q'\cap M.$ If $P$ is a regular subfactor of $M$ then
$P\prec_M N$.
\end{thm}

\begin{proof}
Applying Proposition \ref{ultraproduct} and following the proof of
Theorem 4.1 in \cite{cps} we see that the deformation $\theta_t$
converges uniformly on the unit ball of $P$, and thus by Theorem 3.1
in \cite{cps} we have that $P\prec_M N\rtimes A^H$ or $P\prec_M
N\rtimes H$.

Following the same argument as Theorem 4.1 \cite{cps} if we assume
that $P\prec_M N\rtimes A^H$ and $P\nprec_M N$ then we get $Q\prec_M
N\rtimes A\wr H_0$ for some finite subgroup $H_0\subset H$. Since
$A$ is amenable and $H_0$ is finite then $N\rtimes A\wr
H_0\lessdot_MN$. So since $Q\prec_M N\rtimes A\wr H_0$ then by
Proposition \ref{intertwineimpliesamenableover} we have $Q\lessdot_M
N\rtimes A\wr H_0$. Then by part 3 of Proposition 2.4 in \cite{OP1}
we have that $Q\lessdot_M N$ contradicting our assumption.

Thus $P\prec_M N\rtimes H$. Therefore, by Theorem
\ref{intertwinedef}, there exists nonzero projections $p\in P, q\in
N\rtimes H,$ a nonzero partial isometry $v\in M,$ and a
*-homomorphism $\varphi:pPp\rightarrow q(N\rtimes H)q$ such that
$vx=\varphi(x)v, \forall x\in pPp$. Furthermore we have that
$v^*v=p$ and $vv^*=\hat{q}\in\varphi(pPp)'\cap qMq$. Also, by Lemma
3.5 in \cite{P1} we know that $pPp$ is a regular subalgebra of
$pMp$.

Then for all $u\in\mathcal{N}_{pMp}(pPp)$ let us calculate:

\begin{eqnarray*}
\varphi(x)vuv^*
&= vxuv^*\\
& = vu(u^*xu)v^*\\
& = vuv^*v(u^*xu)v^*\\
& = vuv^*\varphi(u^*xu)vv^*\\
& = vuv^*\varphi(u^*xu)
\end{eqnarray*}

Now assume that $P\nprec_M N$, then by part (2) of Lemma 2.4 in
\cite{cps} we have that $vuv^*\in N\rtimes H$. Since $pPp$ is
regular in $pMp$ we would then get that $M\prec_M N\rtimes H$.
However, this is impossible since the fact that $A$ is nontrivial
implies that $[M:N\rtimes H]=\infty$.

\end{proof}

%%%%%%%%%%%%%%%%%%%%%%%%%%%%%%%%%%%%%%%%%%%%%%%%%%%%%%

\section{Proof of Main Theorems}

In this section we prove our main theorem. Our main technical tool
is the following, which is proposition 2.7 in \cite{PV11}. Before we
state the result let us recall that two von Neumann subalgebras
$M_1, M_2\subset M$ of a finite von Neumann algebra $M$ are said to
form a commuting square if $E_{M_1}E_{M_2} = E_{M_2}E_{M_1}$.

\begin{thm}[Popa-Vaes, \cite{PV11}]\label{PVlemma}Let $(M, \tau)$ be a tracial von Neumann algebra with von Neumann subalgebras
$M_1,M_2\subset M$. Assume that $M_1$ and $M_2$ form a commuting
square and that $M_1$ is regular in $M$. If a von Neumann subalgebra
$Q\subset pMp$ is amenable relative to both $M_1$ and $M_2$, then
$Q$ is amenable relative to $M_1\cap M_2.$
\end{thm}

Notice that this theorem allows us to eliminate the case where $Q$
is amenable over $M_1$. More specifically we have the following
observation.

\begin{prop}\label{intersect}Let $G_1$ and $G_2$ be groups.  Let $A$ be a finite amenable von Neumann algebra with
an action of $G_1\times G_2$, and let $Q\subset A\rtimes G_1\times
G_2$ be a nonamenable subalgebra. Then there exists an $i$ such that
$Q$ is not amenable over $A\rtimes G_i$.
\end{prop}

\begin{proof}
If we let $A\rtimes G_i=M_i$ then it is easy to see that $M_1,
M_2\subset M$ form a commuting square. So if $Q$ is amenable over
both $M_i$ we would have that it would be amenable over the
intersection, which is $A$. Since $A$ is amenable this would imply
that $Q$ is amenable.
\end{proof}

Finally combining the above results we can prove our main theorem
(Theorem \ref{maintheorem}).

\begin{proof}
First let us mention that for the case $n=1$, this is equivalent to
the primeness of II$_1$-factors arising from Bernoulli shifts, which
was proven in \cite{P-gap}.

Now notice that we can write $M$ as $M=N_i\rtimes_\sigma A_i\wr
H_i$, where $N_i=L(A_1\wr H_1) \overline{\otimes}
\dots\overline{\otimes}L(A_{i-1}\wr
H_{i-1})\overline{\otimes}L(A_{i+1}\wr H_{i+1})\overline{\otimes}
\dots\overline{\otimes} L(A_{n}\wr H_{n})$ and $\sigma$ is the
trivial action.

%If $n=1$ then $k=1$ and $L(A_1 \wr H_1) = L(Q_1)$ because $L(A_1 \wr
%H_1)$ is Prime.  Could it be the case that $L(A_1 \wr H_1) \simeq
%L(Q_1)$ and $L(Q_1) \otimes \dots \otimes L(Q_k) \simeq \M_m$? Or is
%this ruled out because the $Q_i$'s are non-gamma $II_1$ factors?

Let us define $\widehat{Q_i}=(Q_i)'\cap
M=Q_1\overline{\otimes}\dots\overline{\otimes}Q_{i-1}\overline{\otimes}Q_{i+1}
\overline{\otimes}\dots\overline{\otimes}Q_k$. Since
$H_i\wr\Gamma_i$ does not have property Gamma for all $i$ this
implies, in particular, that $Q_1$ is non-amenable. By proposition
\ref{intersect}, where we let $A=\mathbb{C}$, we know that there is
an $i$ such that $Q_1$ is not amenable over $N_i$

Since $\widehat{Q_1}$ is a regular subalgebra of $M$, then by
Theorem \ref{nAmenOverEmbedd} we get that $\widehat{Q_1}\prec_M N$.

We complete the argument by following Proposition 12 and the
induction argument of the proof of Theorem 1 in \cite{OP03}.

\end{proof}

Before we prove our final theorem let us recall the following
definition:

\begin{defn}
We say that two group $\Gamma$ and $\Lambda$ are \emph{measure
equivalent}, $\Gamma\simeq_{ME}\Lambda$ is there is a diffuse
abelian von Neumann algebra, $A$, and free ergodic trace preserving
actions, $\sigma, \rho$ of $\Gamma$ and $\Lambda$, respectively,
such that $A\rtimes_\sigma\Gamma\simeq (A\rtimes_\rho\Lambda)^t$,
and the isomorphism takes $A$ onto $A^t$.
\end{defn}

With this definition we can now prove our final result (Corollary
\ref{maintheorem2}.)

\begin{proof}
Let $A_1\wr H_1, ..., A_n\wr H_n$ be as above, and let $K_1, ...,
K_m$ be groups. Since $A_1\wr H_1\times\dots\times A_n\wr
H_n\simeq_{ME} K_1\times\dots\times K_m$ and $A_i\wr H_i$ is
nonamenable for all $i$, then $K_j$ is nonamenable for all $j$.

Now we know that there are actions on $L^\infty (X)$ such that
$M=L^\infty(X)\rtimes A_1\wr H_1\times\dots\times A_n\wr
H_n\simeq(L^\infty(X)\rtimes K_1\times\dots\times K_m)^t$. We may
assume that $t=1$.

Let $N_i=A\rtimes A_1\wr H_1\times\dots A_{i-1}\wr H_{i-1}\times
A_{i+1}\wr H_{i+1}\times\dots\times A_n\wr H_n$, so that we have
$M=N_i\rtimes A_i\wr H_i$. As in the proof of the previous theorem,
since $K_i$ is nonamenable, there is an $i$ such that $L(K_1)$ is
nonamenable over $N_i$. Now by the proof of Theorem
\ref{nAmenOverEmbedd} this implies that $L(K_1)'\cap
M=L(K_2\times\dots\times K_m)\prec N_i\rtimes H_i$. Thus by Lemma
2.2 in \cite{cps} we have that $A\rtimes K_2\times\dots\times
K_m\prec N_i\rtimes H_i$. Now since $A\rtimes K_2\times\dots\times
K_m$ is a regular subalgebra we have by Theorem
\ref{nAmenOverEmbedd} that $A\rtimes K_2\times\dots\times K_m\prec
N_i$.

Notice that now we can follow exactly as in the proof of Corolarry C
in \cite{CS10} to get our desired result.

\end{proof}


\begin{thebibliography}{CFW81}\setlength{\itemsep}{-1mm} \setlength{\parsep}{0mm} \small


\bibitem[CH08]{C-H} {\sc I. Chifan and C. Houdayer}, Bass-Serre rigidity results in von Neumann algebras.
{\it Duke Math. J.} {\bf 153} (2010), 23�-54.

\vskip .05in

\bibitem[CPS11]{cps} {\sc I. Chifan and S. Popa and J. O. Sizemore}, Some OE and $W^*$-rigidity
 results for actions by wreath product groups. (2011), {\it Preprint}, arXiv:1110.2151v1.

\vskip .05in

\bibitem[CS10]{CS10} {\sc I. Chifan and T. Sinclair}, On the structural theory of $\rm{II}_1$ factors of negatively curved groups
 (2010), {\it Preprint}, arXiv:1103.4299v2

\vskip .05in

\bibitem[Co76]{C76} {\sc A. Connes},
Classification of injective factors. Cases $\rm{II}_{1},$
$\rm{II}_{\infty },$ $\rm{III}_{\lambda },$ $\lambda \neq1$. {\it
Ann. of Math. (2)} {\bf 104} (1976), 73-115.

\vskip .05in


\bibitem[Ge98]{Ge98} {\sc L. Ge}, Applications of free entropy to finite von Neumann algebras
\textrm{II}.
 {\it Ann. of Math}
{\bf 147} (1998), 143-157.

\vskip .05in

\bibitem[Io06]{Io06} {\sc A. Ioana},
Rigidity results for wreath product II$_1$ factors. {\it J. Funct.
Anal.} {\bf 252} (2007), 763--791.

\vskip .05in

\bibitem[IPP05]{IPP} {\sc A.\ Ioana, J.\ Peterson and S.\ Popa},
Amalgamated free products of weakly rigid factors and calculation of
their symmetry groups. {\it Acta Math.} {\bf 200} (2008), 85--153.

\vskip .05in

\bibitem[MS06]{MS06} {\sc N. Monod and Y. Shalom},
Orbit equivalence rigidity and bounded cohomology. {\it  Ann. of
Math. (2)} {\bf 164} (2006), 825�878.

\vskip .05in

\bibitem[Oz03]{Oz} {\sc N. Ozawa}, Solid von Neumann algebras.
{\it Acta Math.} {\bf 192} (2004), 111-117.

\vskip .05in

\bibitem[OP03]{OP03} {\sc N. Ozawa and S. Popa}, Some prime factorization results for type {${\rm II}\sb 1$}
              factors. {\it Invent. Math.} {\bf 156} (2004)
223--234.

\vskip .05in

\bibitem[OP07]{OP1} {\sc N. Ozawa and S. Popa}, On a class of II$_1$ factors with at
most one Cartan subalgebra. {\it Ann. Math.} {\bf 172} (2010)
713�749.

\vskip .05in



\bibitem[Pe06]{Pe06} {\sc J. Peterson}, {$L^2$}-rigidity in von {N}eumann algebras. \emph{Invent. Math.} \textbf{175} (2006),
              417--433.

\vskip .05in

\bibitem[Po83]{Po83} {\sc S. Popa}, Orthogonal pairs of {$\ast $}-subalgebras in finite von
              Neumann algebras. \emph{J. Operator Theory} \textbf{9} (1983),
              253--268.

\vskip .05in

\bibitem[Po01]{Po01a} {\sc S. Popa}, Some rigidity results for non-commutative Bernoulli shifts.
\emph{J. Funct. Anal.} \textbf{230} (2006), 273--328.

\vskip .05in

\bibitem[Po03]{P1} {\sc S. Popa}, Strong rigidity of II$_1$ factors arising from malleable actions of $w$-rigid groups, Part
I. \emph{Invent. Math.} \textbf{165} (2006), 369--408.

\vskip .05in


\bibitem[Po06a]{P-gap} {\sc S. Popa}, On the superrigidity of malleable actions with spectral gap.
{\it J. Amer. Math. Soc.} {\bf 21} (2008), 981--1000.

\vskip .05in


\bibitem[Po09]{Po09} {\sc S. Popa}, On the classification of inductive limits of $\rm{II}_1$-factors with spectral gap.
{\it Trans. AMS}, {\it To appear}.

\vskip .05in

\bibitem[PV11]{PV11} {\sc S. Popa and S. Vaes}, Unique Cartan decomposition for $\rm{II}_1$ factors arising from arbitrary actions of free groups
 (2011), {\it Preprint},  arXiv:1111.6951v1

\vskip .05in


\bibitem[Sa09]{Sa09} {\sc H. Sako}, Measure equivalence rigidity and bi-exactness of groups.
{\it J. Funct. Anal.} {\bf 257} (2009), 3167�3202.


\end{thebibliography}
\end{document}